\documentclass[a4paper,12pt,reqno]{amsart}

\usepackage[normalem]{ulem}

\usepackage{fullpage}
\usepackage{amssymb}
\usepackage{amsmath}
\usepackage{mathtools}
\usepackage{enumitem}
\usepackage{mathrsfs}
\usepackage[all]{xy}
\usepackage{mathtools}
\usepackage{hyperref}
\usepackage{url}
\usepackage{bbm}
\usepackage{longtable}
\usepackage{dsfont}
\usepackage{upgreek}
\usepackage{lmodern}
\usepackage[OT2, T1]{fontenc}
\usepackage[utf8]{inputenc}

\DeclareSymbolFont{cyrletters}{OT2}{wncyr}{m}{n}
\usepackage[british]{babel}

\usepackage[margin=0.75in]{geometry}
\numberwithin{equation}{section} \numberwithin{figure}{section}

\DeclareSymbolFont{cyrletters}{OT2}{wncyr}{m}{n}
\DeclareMathSymbol{\Sha}{\mathalpha}{cyrletters}{"58}
\DeclareMathSymbol{\Be}{\mathalpha}{cyrletters}{"42}

\newcommand\E{\mathbb{E}}
\renewcommand\P{\mathbb{P}}
\newcommand\Z{\mathbb{Z}}
\newcommand\N{\mathbb{N}}

\newcommand\R{\mathbb{R}}

\renewcommand{\b}{\mathbf}

\renewcommand{\leq}{\leqslant}
\renewcommand{\geq}{\geqslant}
\renewcommand{\#}{\sharp}

\newtheorem{lemma}{Lemma}

\newtheorem{theorem}[lemma]{Theorem}

\newtheorem{corollary}[lemma]{Corollary}

\theoremstyle{definition}

\usepackage[usenames,dvipsnames]{color}

\numberwithin{lemma}{section}


\title{Higher Moments for Polynomial Chowla}
\author{Cameron Wilson}
\address{Mathematics Department    \\ Glasgow University
\\   Glasgow   \\ G12 8QQ  \\  UK}  
\email{c.wilson.6@research.gla.ac.uk}
\date{\today}

\begin{document}

\begin{abstract}
    Let $\lambda(n)$ be the Liouville function. We study the distribution of
    \[
    \frac{1}{x^{1/2}}\sum_{x\leq n\leq 2x}\lambda(f(n))
    \]
    over random polynomials $f$ of fixed degree $d$ and coefficients bounded in magnitude by $H$. In particular we prove that the first $d+1$ moments are Gaussian.
\end{abstract}

\maketitle

\setcounter{tocdepth}{1}
\tableofcontents

\section{Introduction}
Let $\lambda(n)$ be the Liouville function, which we extend to all integers by setting $\lambda(0)=0$ and $\lambda(-n)=\lambda(n)$ for all $n\in\N$. Suppose $f\in\Z[t]$ is a polynomial which is not of the form $f(x)=cg^2(x)$ for some constant $c\in \Z$, and some $g\in\Z[t]$. In 1965, Chowla \cite{Chowla} conjectured that
\begin{equation}\label{Chowlaconjecture}
C_f(x) \coloneqq \sum_{x\leq n\leq 2x}\lambda(f(n)) = o(x).
\end{equation}
When $f(n)=an+h$ for some coprime integers $a,h$ this is equivalent to the Dirichlet's Theorem on primes in arithmetic progressions, but otherwise remains an open problem. More recently, various approximations to this conjecture have been studied: one may consider logarithmic averages of $\lambda(f(n))$, or study the conjecture for polynomials of particular shape, such as those which factorise as the product of $k$ linear factors. Alternatively we can ask whether $\lambda(f(n))$ changes sign infinitely. Significant progress in such problems have been made within the last decade in a variety of works due to K. Matom\"aki, M. Radziwi\l\l, T. Tao and J. Ter\"av\"ainen \cite{MRT1,MR1,MRT2,Tao1,TaoJoni0,TaoJoni1,TaoJoni2} and most recently in work of H. A. Helfgott and M. Radziwi\l\l \cite{HR}, C. Pilatte \cite{P} and J. Ter\"av\"ainen \cite{Joni1}. Connections of conjecture \eqref{Chowlaconjecture} to other problems are discussed in Section $3$ of \cite{Joni1}.

In this paper we consider the behaviour of $C_f(x)$ for random polynomials $f$ of fixed degree. For $f\in\Z[t]$, we use the notation $\|f\|$ to be the maximum absolute value of the coefficients of $f$. Then formally, we aim to study the distribution of $C_f(x)$ for polynomials of fixed degree such that $\|f\|\leq H$, as $H$ grows to infinity. This idea was first considered in work of Browning, Sofos and Ter\"av\"ainen \cite{TimEfJoni} who showed that if $d\geq 1$, $A\geq 1$ and $0< c<\frac{5}{19d}$ are fixed, then for sufficiently large $H$ there are at most $O\left(\frac{H^{d+1}}{(\log H)^{A}}\right)$ degree $d$ polynomials with $\|f\|\leq H$ such that
\[
\sup_{x\in[H^c,2H^c]}\frac{1}{x}\left|\sum_{1\leq n\leq x}\lambda(f(n))\right| > \frac{1}{(\log H)^A}.
\]
This result was obtained by bounding the second moment of $\frac{C_f(x)}{x}$ uniformly in $x\in[H^{c},2H^{c}]$ and applying Chebychev's inequality. The virtue of studying random polynomials is that one obtains some information on $C_f(x)$ of polynomials of high degree, for which conjecture \eqref{Chowlaconjecture} is completely wide open.\\

Our main result shows that for random polynomials of degree $d$, the first $d+1$ moments of $C_f(x)$ follows the Gaussian distribution. In particular, let $\mathcal{C}_k$ denote the $k$-th moment for the standard normal distribution. Then we have the following.

\begin{theorem}\label{MAINTHEOREM}
    Fix $A,\delta>0$ and $k,d\in\N$ such that $1\leq k\leq d+1$. Then for $x,H\geq 2$ such that $x\leq (\log H)^{\delta}$,
    \[
    \sum_{\substack{f\in\Z[t]\\ \deg(f)=d, \|f\|\leq H}}\left[\frac{1}{x^{1/2}}\sum_{x\leq n\leq 2x}\lambda(f(n))\right]^k = \mathcal{C}_k 2^{d+1} H^{d+1}\left(1+O_k\left(\frac{1}{x}\right)\right) + O_{A,d,k}\left(\frac{H^{d+1}}{(\log H)^{A}}\right)
    \]
    where the first implied constant depends only on $k$ and the second depends only on $A$, $d$ and $k$.
\end{theorem}

The main attractive feature of this result is that it provides information on $\frac{C_f(x)}{x^{1/2}}$ as opposed to $\frac{C_f(x)}{x}$, suggesting a stronger bound than \eqref{Chowlaconjecture} is true on average for random polynomials $f$.

Our methods differ greatly from those of Browning, Sofos and Ter\"av\"ainen, which use information on the equidistribution of $\lambda(n)$ in arithmetic progressions in short intervals. Instead we will employ a version of the circle method similar to that used in Section $3$ of \cite{AlexeiandEf}. This method actually allows us to prove the following more general version the Theorem \ref{MAINTHEOREM}.

\begin{theorem}\label{GENERALTHEOREM} Fix $A,\delta>0$, $s\in\N$ and $\b{d}\in\N^s$. Define $D=\sum_{i=1}^{s}(d_i+1)$ and $\mathfrak{d}=\max\{d_1,\ldots,d_s\}$. Finally, suppose $k\in\N$ is such that $1\leq k\leq \mathfrak{d}+1$. Then for $x,H\geq 2$ such that $x\leq (\log H)^{\delta}$,
\[
\mathop{\sum\cdots\sum}_{\substack{f_1,\cdots,f_s\in\Z[t]\\ \deg(f_i)=d_i,\|f_i\|\leq H \\\forall 1\leq i\leq s}}\left[\frac{1}{x^{1/2}}\sum_{x\leq n\leq 2x}\prod_{i=1}^{s}\lambda(f_i(n))\right]^{k} = \mathcal{C}_{k}2^DH^D\left(1+O_k\left(\frac{1}{x}\right)\right) + O_{A,\mathfrak{d},k}\left(\frac{H^D}{(\log H)^{A}}\right)
\]
where $\mathcal{C}_k$ is as defined in Theorem \ref{MAINTHEOREM}, and where the first implied constant depends only on $k$ and the second depends only on $A$, $\mathfrak{d}$ and $k$.
\end{theorem}

Recall that random polynomials are irreducible with probability $1$. Therefore Theorem \ref{MAINTHEOREM} is mostly relevant to the Chowla conjecture for a single irreducible polynomial. On the other hand Theorem \ref{GENERALTHEOREM} covers polynomials with arbitrary factorisation.

From Theorem \ref{MAINTHEOREM}, one may apply higher moment methods to obtain the following.

\begin{corollary}\label{lowerboundsonavg}
    Fix $d\in\N_{\geq 3}$ and $\delta>0$. Suppose that $X=X(H)\leq (\log H)^{\delta}$ tends to infinity with $H$ and that $y\leq \left(\frac{d-2}{8}\right)^{\frac{d-2}{2(d+1)}}$. Then, for sufficiently large $H$, a positive proportion of degree $d$ polynomials $f\in\Z[t]$ with $\|f\|\leq H$ satisfy, $$\frac{1}{X^{1/2}}\left|\sum_{X\leq n\leq 2X}\lambda(f(n))\right|>y.$$
\end{corollary}

\begin{corollary}\label{upperboundonaverage}
    Fix $d\in\N$ and $\delta>0$. Suppose that $(\log H)^{\delta/2} \leq x \leq (\log H)^{\delta}$ and $0<y\leq x$. Then for all but at most $O_d\left(\frac{H^{d+1}}{y^2}\right)$ polynomials of degree $d$ with $\|f\|\leq H$, we have $$\frac{1}{x^{1/2}}\left|\sum_{x\leq n\leq 2x}\lambda(f(n))\right|\leq y.$$
\end{corollary}

Next, we note that we may obtain results for sums of $\lambda(f(n))$ over the primes:
\[
P_f(x) \coloneqq \sum_{\substack{x\leq p\leq 2x\\ p\;\textrm{prime}}} \lambda(f(p)).
\]
This will be a consequence of the following theorem.
\begin{theorem}\label{SECONDARY}
Fix $A,\delta>0$, $s\in\N$ and $\b{d}\in\N^s$. Define $D=\sum_{i=1}^{s}(d_i+1)$ and $\mathfrak{d}=\max\{d_1,\ldots,d_s\}$. Suppose that $k\in\{1,2\}$ and that $(b_{n})_{n\in\N}$ is an arbitrary complex sequence bounded in magnitude by $1$. Then for $x,H\geq 2$ such that $x\leq (\log H)^{\delta}$,
\[
\mathop{\sum\cdots\sum}_{\substack{f_1,\cdots,f_s\in\Z[t]\\ \deg(f_i)=d_i,\|f_i\|\leq H \\\forall 1\leq i\leq s}}\hspace{-5pt}\left[\frac{1}{x^{1/2}}\sum_{x\leq n\leq 2x}b_n\prod_{i=1}^{s}\lambda(f_i(n))\right]^{k} = \frac{\mathds{1}(k=2)}{x}\left(\sum_{x\leq n\leq 2x}|b_n|^2\right)2^DH^D + O_{A,\mathfrak{d}}\left(\frac{H^D}{(\log H)^{A}}\right)
\]
where the implied constant depends only on $A$ and $\mathfrak{d}$.
\end{theorem}
By using Chebychev's inequality and applying this theorem with $s=1$ and $b_n = \mathds{1}(n\;\textrm{prime})$, we arrive at the following result, which states that $P_f(x)$ exhibits better than square-root cancellation for average polynomials.
\begin{corollary}
     Fix $d\in\N$ and $\delta>0$. Suppose that $(\log H)^{\delta/2} \leq x \leq (\log H)^{\delta}$ and $0<y\leq x$. Then for all but at most $O_d\left(\frac{H^{d+1}}{y^2}\right)$ polynomials of degree $d$ with $\|f\|\leq H$, we have $$\frac{1}{x^{1/2}}\left|\sum_{\substack{x\leq p\leq 2x\\ p\;\textrm{prime}}}\lambda(f(p))\right|\leq \frac{y}{(\log x)^{1/2}}.$$
\end{corollary}

\subsection{Acknowledgements} The author thanks his Ph.D supervisor Efthymios Sofos for suggesting this problem, and the Carnegie Trust for the Universities of Scotland for the funding obtained through their Ph.D Scholarship programme.

\section{Multiple point correlation}
In this section we state some lemmas which will play key roles in the proof of Theorem \ref{GENERALTHEOREM}. We begin with a well known result of Davenport \cite{Davenport}:

\begin{lemma}[Davenport]
Fix $A>0$. Then for all $x\geq 1$ we have:
\[
\sup_{\alpha\in\R}\left|\sum_{n\in[1,x]\cap\N}\mu(n)e^{in\alpha}\right| \ll_A \frac{x}{(\log x)^A}
\]
where the implied constant depends only on $A>0$.
\end{lemma}

By using the M\"obius inversion identity

\begin{equation*}
\lambda(n) = \sum_{d^2|n}\mu\left(\frac{n}{d^2}\right)
\end{equation*}
we may obtain:

\begin{corollary}\label{Davenportforlambda}
Fix $A>0$. Then for all $x\geq 1$ we have:
\[
\sup_{\alpha\in\R}\left|\sum_{n\in[1,x]\cap\N}\lambda(n)e^{in\alpha}\right| \ll_A \frac{x}{(\log x)^A}
\]
where the implied constant depends only on $A>0$.
\end{corollary}

Moving forward we will define

\[
S(\alpha,X) = \sum_{n\in[-X,X]\cap\Z} \lambda(n)e^{in\alpha}.
\]
We will also make use of the Dirichlet kernel $D_H(\alpha)$ which is defined by

\[
D_H(\alpha) = \sum_{n\in[-H,H]\cap\Z}e^{in\alpha}.
\]
Of particular use will be the following result of Lebesgue \cite[Eq.(12.1), pg.67]{Lebesguesref}:

\begin{lemma}[Lebesgue]\label{DKBOUND}
    Let $H\geq 2$. Then
    \[
    \int_{-\pi}^{\pi}\lvert D_H(\alpha)\rvert \mathrm{d}\alpha \ll \log H
    \]
    where the implied constant is absolute.
\end{lemma}

The following two lemmas will be the key ingredients to the proof of Theorem \ref{GENERALTHEOREM}. The first is the activation of the circle method.

\begin{lemma}\label{circlemethod}
    Let $H\geq 2$ and let $L,d\in\N$ be such that $1\leq L\leq d+1$. Suppose $m_1,\ldots,m_L$ are positive integers. Then
    \[
    \sum_{\substack{f\in \Z[t],\mathrm{deg}(f)\leq d \\ \|f\|\leq H}} \left[\prod_{i=1}^{L}\lambda(f(m_i))\right] = \frac{1}{(2\pi)^L}\int_{(-\pi,\pi]^L} \left[\prod_{i=1}^{L}\overline{S(\alpha_i,(d+1)m_i^d H)}\right] \left[\prod_{j=0}^{d}D_H\left(\sum_{i=1}^{L}m_i^j\alpha_i\right)\right] \mathrm{d}\boldsymbol{\alpha}.
    \]
\end{lemma}

\begin{proof}
    Writing $\mathds{1}(r=s)$ to be the indicator function for the condition $r=s$, we have 
    \[
    \sum_{\substack{f\in \Z[t],\mathrm{deg}(f)\leq d \\ \|f\|\leq H}} \left[\prod_{i=1}^{L}\lambda(f(m_i))\right] = \mathop{\sum\cdots\sum}_{\substack{s_1,\ldots s_L\in\Z\\ |s_i|\leq 
    (d+1) m_i^d H}} \left[\prod_{i=1}^{L}\lambda(s_i)\right]\sum_{\substack{f\in \Z[t],\mathrm{deg}(f)\leq d \\ \|f\|\leq H}} \left[\prod_{i=1}^{L}\mathds{1}(f(m_i)=s_i)\right].
    \]
    Then, by using the identity,
    \[
    \mathds{1}(r=s) = \frac{1}{2\pi}\int_{(-\pi,\pi]}e^{i(s-r)\alpha}\mathrm{d}\alpha
    \]
    for each of the $s_i$ and writing $f(m_i) = \sum_{j=0}^{d}c_jm_i^j$ we obtain
    \[
    \sum_{\substack{f\in \Z[t],\mathrm{deg}(f)\leq d \\ \|f\|\leq H}} \left[\prod_{i=1}^{L}\mathds{1}(f(m_i)=s_i)\right] = \frac{1}{(2\pi)^L}\int_{(-\pi,\pi]^L}\left[\prod_{i=1}^{L}e^{-is_i\alpha_i}\right]\mathop{\sum\cdots\sum}_{\substack{c_0,\ldots,c_d\in\Z\\ |c_j|\leq H}} \left[\prod_{i=1}^{L}e^{i\left(\sum_{j=0}^{d}c_j m_i^j\right)\alpha_i}\right]\mathrm{d}\boldsymbol{\alpha}.
    \]
    Rearranging the $d+1$-fold sum inside the integral will give the product of Dirichlet kernels, and so, by swapping the sum over the $s_i$ and the integral, we obtain:
    \[
    \sum_{\substack{f\in \Z[t],\mathrm{deg}(f)\leq d \\ \|f\|\leq H}} \left[\prod_{i=1}^{L}\lambda(f(m_i))\right] = \frac{1}{(2\pi)^L}\int_{(-\pi,\pi]^L}\mathop{\sum\cdots\sum}_{\substack{s_1,\ldots s_L\in\Z\\ |s_i|\leq 
    (d+1) m_i^d H}}\hspace{-6pt}\left[\prod_{i=1}^{L}\lambda(s_i)e^{-is_i\alpha_i}\right]\hspace{-5pt}\left[\prod_{j=0}^{d}D_H\left(\sum_{i=1}^{L}m_i^j\alpha_i\right)\right]\mathrm{d}\boldsymbol{\alpha}.
    \]
    The sum over the $s_i$ is then clearly
    \[
    \prod_{i=1}^{L}\overline{S(\alpha_i,(d+1)m_i^dH)},
    \]
    which concludes the proof.
\end{proof}

In light of this lemma we will define
\[
I_{d,L}(H,\b{m}) = \frac{1}{(2\pi)^L}\int_{(-\pi,\pi]^L} \left[\prod_{i=1}^{L}\overline{S(\alpha_i,(d+1)m_i^d H)}\right] \left[\prod_{j=0}^{d}D_H\left(\sum_{i=1}^{L}m_i^j\alpha_i\right)\right] \mathrm{d}\boldsymbol{\alpha}
\]
for $H\geq 2$, integers $d,L$ such that $1\leq L\leq d+1$ and $\b{m}\in \N^{L}$. Our next lemma bounds these integrals, given some necessary assumptions on $\b{m}$.

\begin{lemma}\label{circlemethodbound}
Let $H\geq 2$ and let $L,d$ be integers such that $1\leq L\leq d+1$. Suppose $\b{m}\in\N^L$ is a vector of pairwise distinct, positive integers. Then for any $A>0$
\[
I_{d,L}(H,\b{m}) \ll_{A,d,L} \frac{\left(\prod_{i=1}^{L}m_i\right)^{\frac{1}{2}(L-1)(L-2)+d}H^{d+1}}{\left(\prod_{1\leq i<j\leq L}|m_i-m_j|\right)(\log H)^{A}}
\]
where the implied constant depends at most on $A,L$ and $d$.    
\end{lemma}

\begin{proof}
Using the triangle inequality we write,
\[
I_{d,L}(H,\b{m}) \leq \frac{1}{(2\pi)^L}\int_{(-\pi,\pi]^L}\prod_{i=1}^{L}\left\lvert S(\alpha_i,(d+1)m_i^d H)\right\rvert\prod_{j=L}^{d}\left\lvert D_H\left(\sum_{i=1}^{L}m_i^j\alpha_i\right) \right\rvert \prod_{j=0}^{L-1}\left\lvert D_H\left(\sum_{i=1}^{L}m_i^j\alpha_i\right) \right\rvert\mathrm{d}\boldsymbol{\alpha}.
\]
We then use the bounds
\[
\left\lvert S(\alpha_i,(d+1)m_i^d H)\right\rvert \ll_{B,d} \frac{m_i^d H}{(\log m_i^d H)^B}\ll_{B,d} \frac{m_i^d H}{(\log H)^B}\;\;\text{and}\;\;\left\lvert D_H\left(\sum_{i=1}^{L}m_i^j\alpha_i\right)\right\rvert \ll H
\]
for the terms in the first and second products respectively. The first bound is obtained from Corollary \ref{Davenportforlambda} for some $B>0$, while the second is just the trivial bound. Doing this will yield,
\[
I_{d,L}(H,\b{m}) \ll_{B,d} \frac{\left(\prod_{i=1}^{L}m_i^d\right)H^{d+1}}{(\log H)^{B}}\frac{1}{(2\pi)^L}\int_{(-\pi,\pi]^L} \prod_{j=0}^{L-1}\left\lvert D_H\left(\sum_{i=1}^{L}m_i^j\alpha_i\right) \right\rvert\mathrm{d}\boldsymbol{\alpha}.
\]
Next we use the change of variables $t_j = \sum_{i=1}^{L} m_i^j \alpha_i$, whose matrix is the $L\times L$ Vandermode matrix, $V_L(\b{m})$, with determinant
\[
\det(V_L(\b{m})) = \prod_{1\leq i<j\leq L}(m_j-m_i)
\]
which is non-zero since we assume the $m_i$ are pairwise distinct. This change of variables therefore yields
\[
I_{d,L}(H,\b{m}) \ll_{B,d} \frac{\left(\prod_{i=1}^{L}m_i^d\right)H^{d+1}}{\lvert \det(V_L(\b{m}))\rvert(\log H)^{B}}\frac{1}{(2\pi)^L}\int_{\prod_{j=0}^{L-1}(-\mathcal{M}_j\pi,\mathcal{M}_j\pi]} \prod_{j=0}^{L-1}\left\lvert D_H\left(t_j\right) \right\rvert\mathrm{d}\b{t}
\]
where we have set $\mathcal{M}_j = \sum_{i=1}^{L}m_i^j$. Next we use the fact that $D_H(t)$ is a $2\pi$-periodic, even function to write
\[
\frac{1}{(2\pi)^L}\int_{\prod_{j=0}^{L-1}(-\mathcal{M}_j\pi,\mathcal{M}_j\pi]} \prod_{j=0}^{L-1}\left\lvert D_H\left(t_j\right) \right\rvert\mathrm{d}\b{t} \ll_{L} \left(\prod_{j=0}^{L-1}\mathcal{M}_j\right)\int_{(-\pi,\pi]^L} \prod_{j=0}^{L-1}\left\lvert D_H\left(t_j\right) \right\rvert\mathrm{d}\b{t}.
\]
Thus, using Lemma \ref{DKBOUND} this becomes
\[
= \left(\prod_{j=0}^{L-1}\mathcal{M}_j\right)\left(\int_{(-\pi,\pi]}\left\lvert D_H\left(t\right) \right\rvert\mathrm{d}t\right)^L \ll  \left(\prod_{j=0}^{L-1}\mathcal{M}_j\right)(\log H)^L.
\]
Thus we have proven that
\[
I_{d,L}(H,\b{m}) \ll_{B,d,L} \frac{\left(\prod_{j=0}^{L-1}\mathcal{M}_j\right)\left(\prod_{i=1}^{L}m_i^d\right)H^{d+1}}{\lvert \det(V_L(\b{m}))\rvert(\log H)^{B-L}}.
\]
Finally, we note that
\[
\prod_{j=0}^{L-1}\mathcal{M}_j = \sum_{\substack{\b{h}\in [0,\frac{1}{2}(L-1)(L-2)]^L\\ \sum_{j=1}^L h_i = \frac{1}{2}(L-1)(L-2)}} c_{\b{h}}\left[\prod_{j=1}^{L}m_j^{h_j}\right] \ll_L \prod_{j=1}^{L}m_j^{\frac{1}{2}(L-1)(L-2)},
\]
where we note that $c_{\b{h}}$ are some constants depending only on $\b{h}$ and whose sum over the given region is therefore only dependent on $L$. Upon setting $B=A+L$ we are therefore done.
\end{proof}

\section{Proof of Theorem \ref{GENERALTHEOREM}}
We begin by noting that the conditions $\deg(f_i)=d_i$ only dictates that the leading coefficient of each $f_i$ is not $0$. The contribution from one of the leading coefficients being $0$ is trivially $O(x^{k/2}H^{D-1})$ which is sufficient for our purposes. Thus we have
\[
\mathop{\sum\cdots\sum}_{\substack{f_1,\cdots,f_s\in\Z[t]\\ \deg(f_i)=d_i,\|f_i\|\leq H \\\forall 1\leq i\leq s}}\hspace{-7.5pt}\left[\frac{1}{x^{1/2}}\sum_{x\leq n\leq 2x}\prod_{i=1}^{s}\lambda(f_i(n))\right]^{k} = \hspace{-5pt}\mathop{\sum\cdots\sum}_{\substack{f_1,\cdots,f_s\in\Z[t]\\ \deg(f_i)\leq d_i,\|f_i\|\leq H \\\forall 1\leq i\leq s}}\hspace{-7.5pt}\left[\frac{1}{x^{1/2}}\sum_{x\leq n\leq 2x}\prod_{i=1}^{s}\lambda(f_i(n))\right]^{k} + O(x^{k/2}H^{D-1})
\]
and we may focus on the right hand side. To do so we first expand the $k$-th power on the right hand side using the multinomial theorem. This will yield
\[
=\mathop{\sum\cdots\sum}_{\substack{f_1,\cdots,f_s\in\Z[t]\\ \deg(f_i)\leq d_i,\|f_i\|\leq H \\\forall 1\leq i\leq s}}\frac{1}{x^{k/2}}\sum_{u=1}^{k}\sum_{\substack{l_1,\ldots,l_u\geq 1\\l_1+\ldots+l_u=k}}\frac{k!}{l_1!\cdots l_u!}\mathop{\sum\cdots\sum}_{x\leq m_1<\ldots<m_u\leq 2x}\prod_{i=1}^{s}\prod_{j=1}^{u}\lambda(f_i(m_j))^{l_j}.
\]
We split the sum over $l_i$ into two cases: one where all $l_j$ are even and a second where at least one $l_i$ is odd. The overall expression is therefore the sum of the following two:
\begin{align*}
N_0(x,H) = \frac{1}{x^{k/2}}\mathop{\sum\cdots\sum}_{\substack{f_1,\cdots,f_s\in\Z[t]\\ \deg(f_i)\leq d_i,\|f_i\|\leq H \\\forall 1\leq i\leq s}}\sum_{u=1}^{k}\sum_{\substack{l_1,\ldots,l_u\geq 1\\l_1+\ldots+l_u=k\\\text{all $l_j$ even}}}&\frac{k!}{l_1!\cdots l_u!}\mathop{\sum\cdots\sum}_{x\leq m_1<\ldots<m_u\leq 2x}1.\\
N_1(x,H) = \frac{1}{x^{k/2}}\mathop{\sum\cdots\sum}_{\substack{f_1,\cdots,f_s\in\Z[t]\\ \deg(f_i)\leq d_i,\|f_i\|\leq H \\\forall 1\leq i\leq s}}\sum_{u=1}^{k}\sum_{\substack{l_1,\ldots,l_u\geq 1\\l_1+\ldots+l_u=k\\\text{at least one}\\ l_j \text{odd}}}&\frac{k!}{l_1!\cdots l_u!}\mathop{\sum\cdots\sum}_{x\leq m_1<\ldots<m_u\leq 2x}\prod_{i=1}^{s}\prod_{j=1}^{u}\lambda(f_i(m_j))^{l_j}.
\end{align*}
We first deal with $N_0(x,H)$. This is $0$ if $k$ is odd since all partitions of odd numbers must contain an odd number. If $k$ is even and all $l_i$ are even, then $u\leq k/2$. Thus we write,
\begin{align*}
N_0(x,H) &= \frac{1}{x^{k/2}}\mathop{\sum\cdots\sum}_{\substack{f_1,\cdots,f_s\in\Z[t]\\ \deg(f_i)\leq d_i,\|f_i\|\leq H \\\forall 1\leq i\leq s}}\sum_{u=1}^{k/2}\sum_{\substack{l_1,\ldots,l_u\geq 1\\l_1+\ldots+l_u=k\\\text{all $l_j$ even}}}\frac{k!}{l_1!\cdots l_u!}\left(\frac{x^u}{u!}+O\left(x^{u-1}\right)\right).
\end{align*}
The contribution from $u<k/2$ is therefore trivially $\ll_k x^{k/2-1}H^{D}$. When $u=k/2$, all of the $l_i$ must be $2$. Thus,
\begin{align*}
N_0(x,H) &= \frac{1}{x^{k/2}}\mathop{\sum\cdots\sum}_{\substack{f_1,\cdots,f_s\in\Z[t]\\ \deg(f_i)\leq d_i,\|f_i\|\leq H \\\forall 1\leq i\leq s}}\frac{k!}{2^{k/2}(k/2)!}\left(x^{k/2}+O\left(x^{k/2-1}\right)\right) + O_k\left(\frac{H^{D}}{x}\right)\\
&=(k-1)!!2^D H^D\left(1 + O_k\left(\frac{1}{x}+\frac{1}{H}\right)\right).
\end{align*}
This will clearly give the desired main term since $x\leq (\log H)^{\delta}$. We are left to bound $N_1(x,H)$. By swapping the orders of summations we have
\[
N_1(x,H) = \frac{1}{x^{k/2}}\sum_{u=1}^{k}\sum_{\substack{l_1,\ldots,l_u\geq 1\\l_1+\ldots+l_u=k\\\text{at least one}\\ l_j \text{odd}}}\frac{k!}{l_1!\cdots l_u!}\mathop{\sum\cdots\sum}_{\substack{x\leq m_1<\ldots<m_u\leq 2x}}\prod_{i=1}^{s}\left[\mathop{\sum}_{\substack{f_i\in\Z[t]\\ \deg(f_i)\leq d_i,\|f_i\|\leq H}}\prod_{\substack{j=1\\l_j\;\text{odd}}}^{u}\lambda(f_i(m_j))\right].
\]
Next we aim to apply Lemmas \ref{circlemethod} and \ref{circlemethodbound} to each of the inner sums, using $$L=L(\b{l})=\#\{1\leq j\leq u:l_j\;\text{odd}\}\leq k.$$ However, since we only obtain arbitrary logarithmic saving from each of these we need only do it on one. For this reason we will assume without loss in generality that $d_1=\mathfrak{d}$ and only apply the lemmas to the sum over $f_1$, as we then have that $k\leq d_1 + 1$. For the other $f_i$ we will use the trivial bound. Thus:
\[
N_1(x,H) \ll_k \frac{H^{D-d_1-1}}{x^{k/2}}\sum_{u=1}^{k}\sum_{\substack{l_1,\ldots,l_u\geq 1\\l_1+\ldots+l_u=k\\\text{at least one}\\ l_j \text{odd}}}\mathop{\sum\cdots\sum}_{\substack{x\leq m_1<\ldots<m_u\leq 2x}}\left\lvert\mathop{\sum}_{\substack{f_1\in\Z[t]\\ \deg(f_1)\leq d_1,\|f_1\|\leq H}}\prod_{\substack{j=1\\l_j\;\text{odd}}}^{|\b{l}|}\lambda(f_i(m_j))\right\rvert.
\]
We may now apply Lemmas \ref{circlemethod} and \ref{circlemethodbound} to the inner sum with $L=L(\b{l})$. This will yield:
\begin{align*}
N_1(x,H) &\ll_{B,d_1,k} \frac{H^{D}}{x^{k/2}(\log H)^{B}}\sum_{u=1}^{k}\sum_{\substack{l_1,\ldots,l_u\geq 1\\l_1+\ldots+l_u=k\\\text{at least one}\\ l_j \text{odd}}}\mathop{\sum\cdots\sum}_{\substack{x\leq m_1<\ldots<m_u\leq 2x}} \left(\prod_{\substack{j=1\\l_j\;\text{odd}}}^{u}m_j\right)^{\frac{1}{2}(L(\b{l})-1)(L(\b{l})-2)+d_1}\\
&\ll_{B,d_1,k} \frac{H^{D}}{x^{k/2}(\log H)^{B}}\sum_{u=1}^{k}\sum_{\substack{l_1,\ldots,l_u\geq 1\\l_1+\ldots+l_u=k\\\text{at least one}\\ l_j \text{odd}}} x^{v(\b{l})}(\log x)^{L(\b{l})-1}.
\end{align*}
where
\[
v(\b{l}) = \frac{1}{2}(L(\b{l})-1)(L(\b{l})-2)+d_1+1+|\b{l}|-L(\b{l}) \leq \frac{1}{2}(k-1)(k-2)+d_1 +1.
\]
Thus, using $x\leq (\log H)^{\delta}$ and using $B = (A + \delta(\frac{1}{2}k^2 - 2k + d_1 + 1)+1)$,
\begin{align*}
N_1(x,H) &\ll_{A,d_1,k} \frac{H^{D}}{(\log H)^{A}}
\end{align*}
as required.

\section{Proof of Corollaries \ref{lowerboundsonavg} and \ref{upperboundonaverage}}
Define $$S_f(x) = \frac{1}{x^{1/2}}\left|\sum_{x\leq n\leq 2x}\lambda(f(n))\right|$$
for $f\in\Z[t]$. Denote by $\P(S_f(x)>y)$ the probability of the event $S_f(x)>y$, by $\E[S_f(x)]$ the expected value of $S_f(x)$ over random polynomials of degree $d$ and coefficients bounded in magnitude by $H$. Denote by $\E_{S_f(x)>y}[S_f(x)]$ the expected value of $S_f(x)$ over those polynomials which satisfy $S_f(x)>y$.

\subsection{Proof of Corollary \ref{lowerboundsonavg}} 
For Corollary \ref{lowerboundsonavg} let us first prove the following claim:
\[
\mathcal{C}_{2m} > \left(\frac{m}{2}\right)^{m}
\]
for all $m\in\N$. Indeed we have
\[
\mathcal{C}_{2m} = (2m-1)!! = \frac{(2m-1)!}{2^{m-1}(m-1)!} = \frac{(2m-1)\cdots(m)}{2^{m-1}}>\frac{m^m}{2^{m-1}}>\frac{m^m}{2^m}
\]
Now take $m$ to be the largest integer such that $4m\leq d+1$. Note that this implies that $2m\geq \frac{d-2}{2}$. We have the following:
\[
\E[S_f(X)^{2m}] = \E_{S_f(X)\leq y}[S_f(X)^{2m}] + \E_{S_f(X)>y}[S_f(X)^{2m}] \leq y^{2m} + \E_{S_f(X)>y}[S_f(X)^{2m}].
\]
Now using the Cauchy--Schwarz inequality on $\E_{S_f(X)>y}[S_f(X)^{2m}]$ we obtain the following:
\[
\P(S_f(X)>y)\geq \frac{\left(\E[S_f(X)^{2m}]-y^{2m}\right)^2}{\E[S_f(X)^{4m}]}.
\]
Using the asymptotics for $2m$-th and $4m$-th moments gives us
\[
\P(S_f(X)>y)\geq \frac{\left(\mathcal{C}_{2m} + O\left(\frac{1}{X}\right) + O_{m,d,A}\left(\frac{1}{(\log H)^{A}}\right)-y^{2m}\right)^2}{\mathcal{C}_{4m} + O\left(\frac{1}{X}\right) + O_{m,d,A}\left(\frac{1}{(\log H)^{A}}\right)}.
\]
Since $X=X(H)$ goes to infinity with $H$ by assumption, we have that, for all $\epsilon>0$, there exists some $H_0(\epsilon,d,A)$ such that for $H\geq H_0$,
\[
\P(S_f(X)>y)\geq \frac{\left(\mathcal{C}_{2m} - \epsilon-y^{2m}\right)^2}{\mathcal{C}_{4m} + \epsilon}.
\]
Therefore, using the conditions $y \leq \left(\frac{d-2}{8}\right)^{\frac{d-2}{2(d+1)}}$, $2m\geq \frac{d-2}{2}$ and $4m\leq d+1$ with the claim above we have
\[
\mathcal{C}_{2m}>\left(\frac{m}{2}\right)^{m} \geq \left(\frac{d-2}{8}\right)^{\frac{d-2}{4}}\geq y^{\frac{d+1}{2}} \geq y^{2m}.
\]
Therefore, choosing $\epsilon>0$ sufficiently small and $H\geq H_0(\epsilon,d,A)$ we will have
\[
\P(S_f(X)>y) > 0.
\]

\subsection{Proof of Corollary \ref{upperboundonaverage}} For Corollary \ref{upperboundonaverage}, we use Chebychev's inequality to obtain
\begin{align*}
\#\{f\in\Z[t]:\deg(f)=d,\;\|f\|\leq H,\;S_f(x)>y\} &\leq (2H)^{d+1} \frac{\E[S_f(x)^2]}{y^2}\\
&\leq \frac{2^{d+1}H^{d+1}}{y^2} + \left(\frac{H^{d+1}}{y^2x}\right) + \left(\frac{H^{d+1}}{y^2(\log H)^A}\right)\\
&\ll \frac{2^{d+1}H^{d+1}}{y^2}.
\end{align*}

\section{Proof of Theorem \ref{SECONDARY}}
For the first moment one only needs to swap the order of summation and apply Lemmas \ref{circlemethod} and \ref{circlemethodbound}. We therefore focus on the second moment. We begin in a similar way as in the proof of Theorem \ref{GENERALTHEOREM}, adding in terms for which the leading coefficients of the polynomials being $0$, expanding the $k$-th power and splitting off the diagonal terms which have no oscillation to obtain,
\[
\mathop{\sum\cdots\sum}_{\substack{f_1,\cdots,f_s\in\Z[t]\\ \deg(f_i)=d_i,\|f_i\|\leq H \\\forall 1\leq i\leq s}}\hspace{-7.5pt}\left[\frac{1}{x^{1/2}}\sum_{x\leq n\leq 2x}b_n\prod_{i=1}^{s}\lambda(f_i(n))\right]^{2} = M_0(x,H) + M_1(x,H)
\]
where
\[
M_0(x,H) = \frac{1}{x}\mathop{\sum\cdots\sum}_{\substack{f_1,\cdots,f_s\in\Z[t]\\ \deg(f_i)\leq d_i,\|f_i\|\leq H \\\forall 1\leq i\leq s}}\mathop{\sum}_{x\leq m\leq 2x}|b_m|^2.
\]
and
\[
M_1(x,H) = \frac{1}{x}\mathop{\sum\cdots\sum}_{\substack{f_1,\cdots,f_s\in\Z[t]\\ \deg(f_i)\leq d_i,\|f_i\|\leq H \\\forall 1\leq i\leq s}}2\mathop{\sum\cdots\sum}_{x\leq m_1<m_2\leq 2x}b_{m_1}\overline{b_{m_2}}\prod_{i=1}^{s}\lambda(f_i(m_1))\lambda(f_i(m_2)).
\]
$M_1(x,H)$ may be handled using the same method as for $N_1(x,H)$ in the proof of Theorem \ref{GENERALTHEOREM} since the complex sequence $b_n$ is bounded in magnitude by $1$. Thus,
\[
M_1(x,H) \ll_{A,\mathfrak{d}} \frac{H^D}{(\log H)^A}.
\]
For the diagonal terms, we clearly have
\[
M_0(x,H) = \frac{1}{x}\left(\sum_{x\leq m\leq 2x}|b_m|^2\right)\left(2^{D}H^{D}+O(H^{D-1})\right).
\]
This is enough to conclude the result.

\end{document}